\documentclass[a4paper,10pt]{article}

\let \theoremstyle \relax

\usepackage{amsmath, amssymb, amsthm, mathrsfs, bm, mathtools,textcomp}

\usepackage{multirow, booktabs,  float,  graphicx, cases, nomencl, listings, cancel, chngpage, emptypage, setspace, color, tabularx}
\usepackage{subcaption}
\usepackage[utf8]{inputenc}
\usepackage{siunitx}
\usepackage{psfrag}
\usepackage[english]{babel}
\usepackage{amsmath, amssymb, amsthm, mathrsfs, bm, mathtools,hyperref}
\usepackage{algorithm}
\usepackage[numbers,square, comma, sort&compress]{natbib}
\usepackage{nomencl}
\usepackage{algpseudocode}
\usepackage{pifont}
\usepackage{subcaption}
\algdef{SE}[DOWHILE]{Do}{doWhile}{\algorithmicdo}[1]{\algorithmicwhile\ #1}%
\theoremstyle{plain}
\usepackage{ctable}

\newtheorem{definition}{Definition}

\newtheorem{lemma}{Lemma}
\newtheorem{theorem}{Theorem}

\newcommand{\E}{E}

\DeclarePairedDelimiter\floor{\lfloor}{\rfloor}
\usepackage{amssymb}
\usepackage{nomencl}
\makenomenclature
\usepackage{etoolbox}
\renewcommand\nomgroup[1]{%
	\item[\bfseries
	\ifstrequal{#1}{A}{Physics Constants}{%
		\ifstrequal{#1}{B}{Number Sets}{%
			\ifstrequal{#1}{C}{Other Symbols}{}}}%
	]}

\title{A multiscale, asymptotically unbiased approach to uncertainty quantification in the numerical approximation of infinite time-averaged statistics}
\author{Pooriya Beyhaghi \and Shahrouz R. Alimo \and Thomas R. Bewley}
\begin{document}
\maketitle
\begin{abstract}
Accurate assessment of uncertainty in the approximation of infinite-time-averaged statistics of statistically stationary ergodic processes (that is, of signals, obtained experimentally or computationally, that meander about some unknown mean) is a topic of significant importance in a host of engineering applications.  Among them, for example, the statistics of many turbulent flows are generally considered as stationary and ergodic after some initial transient is identified and set aside. As taking infinite time averages is not practically feasible, finite-time-averaged approximations of these statistics are generally used.  For problems in which the measured samples are independent and identically distributed (i.i.d.), the expected squared averaging error reduces only like $\sigma/\sqrt{N}$, where $\sigma$ is the standard deviation of the data.  In problems for which the measured samples are not i.i.d., convergence is even slower (eventually, like $Q/\sqrt{N}$ for $Q<\sigma$), and an uncertainty quantification (UQ) method is needed.  The present paper presents a new method to quantify the expected squared averaging error which is \emph{multiscale}, meaning that it is based on an autocorrelation model that is tuned to the data to fit the statistic of interest at a large range of different timescales.  The method is also \emph{asymptotically unbiased}, meaning that the expected squared averaging error asymptotically
converges like $Q/\sqrt{N}$ for the same value of $Q$ as the actual system, if it is modelled as a random process with the same mean, variance, and autocorrelation.  The new UQ method is tested on three representative test problems, and shown to be highly effective.
\end{abstract}

\textbf{keywords:}  Uncertainty quantification, time averaging, turbulence statistics.

\section{Introduction} \label{sec:intro}
Statistical characterizations are essential in many scientific and engineering problems.
For example, statistical measures such as turbulent kinetic energy (TKE), skin friction drag, pressure drop, and velocity correlation lengths are of fundamental importance in characterizing turbulent flowfield fluctuations in a time-averaged sense or, if the system is not statistically stationary, in an ensemble-averaged sense (without loss of generality, the present paper focuses on time averaging in the statistically stationary setting).
In practice, only a finite number of samples are available in order to approximate such statistics;
it is thus important to quantify the expected deviation between the quantity measured, obtained with a finite number of samples, and the
infinite-time-averaged statistic of interest.
This quantity is often referred to as {\it time-averaging error}, but is sometimes referred to in the turbulence literature as {\it random error} \cite{salesky-2012} or {\it sampling error} \cite{moser-2014};
for simplicity, the remainder of this paper simply calls it {\it averaging error}.

Estimation of the averaging error plays a key role in determining necessary run times in the large-eddy simulation (LES) and direct numerical simulation (DNS) of turbulent flows of engineering interest; such simulations are often extremely computationally expensive.  The importance of determining the averaging error is especially pronounced in optimization problems, such as shape optimization \cite{mardsen2004}, using derivative-free optimization approaches, as such optimization codes perform and compare many different simulations or experiments, at a variety of different sets of feasible values of the design parameters, in search of the optimum point in parameter space.  In \cite{beyhaghi-alpha}, a derivative-free optimization algorithm has been developed specifically for the efficient minimization of time-averaged statistics of the type considered in this paper.  Accurate uncertainty quantification is of fundamental importance in the effectiveness of such an approach, which adjusts the amount of sampling associated with each individual function evaluation, making function evaluations more accurate (and, thus, more expensive), as required, as convergence is approached.

As mentioned in the abstract, if the measured samples are i.i.d.~(such as thermocouple measurements, for which the noise is often well modelled as white), the standard deviation, $\varepsilon_N$, of a finite-sample approximation of an infinite-time-average statistic (that is, the ``averaging error'') is given by
\begin{subequations}
\begin{equation}  \label{eq:sigmaN}
\varepsilon_N=\sigma \sqrt{1/N},
\end{equation}
where $\sigma$ is the standard deviation of a single sample, and $N$ is the number of samples taken.  However, in most problems of interest, the measured samples are not i.i.d.; in such problems, convergence is even slower.

In a numerical simulation of a continuous-time chaotic system that is assumed to behave in a stationary ergodic manner, such as a turbulent flow, with sampling times of $t_k=k\,h$ for some sampling interval $h$, it is actually the total simulation time $T$, not the total number of samples taken $N$, that best represents the computational expense of a given measurement.  There are four main approaches available in the literature for estimating the averaging error in such problems.  The first, developed in \cite{lumley-1964}, imposes the following informative model for the standard deviation of the average over a simulation time $T$, assuming essentially continuous sampling:
\begin{equation} \label{eq:sigmaT1N}
\hat\varepsilon(T)=\sigma \sqrt{2\,\tau_f/T},
\end{equation}
\end{subequations}
where $\sigma$ is the standard deviation (from the infinite-time average) of a single sample, $\hat\varepsilon(T)$ is a model of the averaging error after time $T$, and $\tau_f$ is a modeling parameter, referred to as the {\it integral time scale}, which is introduced to model the largest decorrelation timescale of the samples of the system.  The integral time scale $\tau_f$ is studied extensively in \cite{wynngaard-1973, sreenivasan-1978}.  This model is found to be effective in practice only if the simulation time divided by the integral time scale, $T/\tau_f$, is relatively large, and is thus of specifically limited utility for UQ; regardless, it is a very revealing starting point, as discussed further in the following paragraph.

In the discrete sampling setting considered in this paper, we take $h$ as the sampling interval, and thus $N=T/h$ as the number of samples taken over a simulation of length $T$.  If $h$ is taken as so large that the samples are effectively decorrelated, \eqref{eq:sigmaN} holds; note that this relation may in fact be recovered by redefining $\tau_f=h/2$ in the discrete realization of the model given in \eqref{eq:sigmaT1N}.  As $h$ is made smaller for a given $T$ and $\tau_f$, more samples are taken, but they begin to become correlated (and, thus, do not each provide independent information).  The relation given by \eqref{eq:sigmaT1N} sets an approximate lower bound on the averaging error over a simulation of length $T$, assuming continuous sampling.
Thus, comparing \eqref{eq:sigmaT1N} and \eqref{eq:sigmaN}, taking $c/N = 2\,\tau_f/T$ where $c$ is some $\sim O(10)$ constant, and noting that $N=T/h$, reveals that
\begin{equation}\label{eq:optimal_h}
h=2\,\tau_f/c
\end{equation}
is a reasonable value for the sampling interval $h$ in a given problem; sampling substantially more frequently than this will not substantially reduce $\varepsilon(T)$, whereas sampling substantially less frequently than this provides less information, thus increasing $\varepsilon(T)$.  Note further that, if $h$ is taken as unnecessarily small (and, thus, $N$ as unnecessarily large), then the overhead associated with storing these samples may become significant, and the errors related to the finite precision of the arithmetic used may corrupt the computation of the average.  This effectively motivates one to pick an appropriate intermediate sampling interval $h$ according to \eqref{eq:optimal_h}, for which samples are indeed correlated with each other.  A well-designed UQ method, such as that designed in the present paper, is thus required to estimate the uncertainty of the approximation of the averaged statistic of interest determined from these samples.

A second approach for estimating the averaging error in such problems is to model the autocorrelation function with a simple exponential decay such that
\begin{equation}
\hat\rho(k)=\exp(- \alpha_f \, k),
\end{equation}
where $\alpha_f$ is a fitting parameter, then using this model to estimate the averaging error.  With this approach, the parameter $\alpha_f$ is determined from the available data via empirical modeling of the autocorrelation function~\cite{alimo-2016}.   
Unfortunately, the data upon which this empirical model of the autocorrelation function is built often has spurious oscillations, which can lead to inaccuracies in the estimation of the averaging error.  Filtering methods have been shown in \cite{Dias-2004, yaglom-1987, theunissen-2008, salesky-2012} to alleviate this problem somewhat, though special care is required in its implementation. 

A third approach for estimating the averaging error is known as batch mean methods (see, e.g., \cite{bernardes-2010, gluhovsky-1994, politis-2004, suarez-2002}). With such methods, the $N$ available samples are divided, in an ad hoc fasion, into $p$ non-overlapping blocks of length $n=N/p$, and the averaged values for these smaller blocks computed to generate another random process. This new process is closer to i.i.d.; the overall averaging error can then be estimated from the nominal deviations of these block averages from their overall average divided by the square root of $p$. A central challenge with this approach is the determination of the block length $n$ that works best for a given sample size $N$.

In the fourth approach for estimating the averaging error, a statistical model of the random process is imposed, where the parameters of this model are determined via a maximum likelihood formulation \cite{beran-1994}. The statistical model which is typically considered in this setting is an autoregressive moving average (ARMA) process (see, e.g., \cite{hosking-1981} and \cite{moser-2014}).
A significant challenge with this approach is the presence of systematic error (a.k.a.~``bias'') in the uncertainty quantification which does not diminish to zero as the simulation time is increased, as quantified further in \S \ref{sec:result} below.

In this paper, we present a new method for quantifying the expected squared averaging error of a finite-time-average approximation of an infinite-time-average statistic of a stationary ergodic process. The method developed (in \S \ref{sec:met}) is \emph{multiscale}, meaning that it is based on an autocorrelation model that is tuned to the data to fit the statistic of interest at a range of different timescales.  The method developed is proven in \S \ref{sec:analysis} to be \emph{asymptotically unbiased} (see Definition \ref{def:asm_unb}), meaning that the expected squared averaging error asymptotically converges like $Q/\sqrt{N}$ for the same value of $Q$ as the actual system, if it is modelled as a random process with the correct (that is, infinite-time-averaged) mean, variance, and autocorrelation.  The maximum likelihood formulation of \cite{beran-1994}, which is a leading competing UQ strategy, is shown to not satisfy this valuable property.  An automated procedure to identify the initial transient in a dataset is also reviewed. A primary application that motivates this work is turbulence research, though many other applications are also envisioned.

The structure of the remainder of the paper is as follows: 
Section \ref{sec:trans} reviews a framework to automatically identify the initial transient of a dataset.
Once this portion of the dataset is set aside, the remainder of the dataset is modeled as a realization of a stationary ergodic process.
Section \ref{sec:met} presents our new method to calculate the averaging error for the stationary part of the dataset.
Section \ref{sec:analysis} analyzes the salient properties of the new method.
Section \ref{sec:result} implements the method developed on synthetic data derived from an autoregressive (AR) model, on data derived from the  Kuramoto-Sivashinsky (KS) equation, and on data derived for a low-Reynolds number turbulent channel flow DNS at $Re_\tau=180$.
%
%
%
%
\clearpage
\subsection{Identification of the initial transient} \label{sec:trans}

In this section, we review three automated procedures to identify approximately the initial transient of a dataset.  As stated previously, this is an important first step in developing an asymptotically unbiased quantification of uncertainty of the average in the applications of interest.

The first approach identifies the smallest transient time such that, by its removal, a second-order stationarity condition is satisfied by the remainder of the dataset.  The second-order stationarity condition may be tested in two different ways:
\begin{itemize} 
\item[a.] the Priestley-Subba-Rao test \cite{priestley-1969}, which is based on a time-varying Fourier spectrum analysis, and
\item[b.] the Wavelet Spectrum test \cite{nason-2013}, which is based on a time-varying wavelet-based analysis.
\end{itemize} 
These two tests are designed to validate or invalidate the stationarity of a given random process, rather than establishing the ``degree'' of stationarity of a dataset, which is perhaps more appropriate for the problems of interest here.

The second approach determines the initial transient based on von Neumann's randomness test \cite{lada-2006, mokashi-2010, robinson-2002, tafazzoli-2011}, which uses a batch means approach which divides, in an ad hoc fashion, the $N$ available samples into $p$ non-overlapping blocks of length $n=N/p$, then analyzes the distribution of the averages of each block. Such a heuristic procedure, which is somewhat computationally expensive, often leads to acceptable results.  Indeed, in problems for which there is a specific time $t_1$ after which the state of the system is {exactly} statistically stationary\footnote{Note that this is not precisely the case in the problems of interest here, in which a continuous-time chaotic system exponentially approaches an attractor.}, the method developed in \cite{awad-2006} is shown to identify $t_1$ correctly in the limit that the simulation time $T$ goes to infinity. 

The third approach, which is implemented in the present work and is computationally quite inexpensive, was originally introduced in \cite{White-1997, White-2000, hoad-2008} for numerical simulations in systems engineering and finance. This approach is well-suited for calculating an unbiased estimate of the infinite time-averaged value of a statistic, as it is specifically designed to find an estimate of the average with minimum uncertainty.  Take $\{X_1, X_2, \dots, X_N\}$ as a dataset modelled as a realization of a random process with $N$ samples; the initial transient of this dataset is estimated via this approach by solving the following optimization problem:
\begin{equation}\label{eq:trans}
\hat k = \underset{1 \leq k \leq {\frac{N}{2}}}{\text{argmin}} \ \frac{1}{(N-k-1)^2} \sum_{i=k+1}^N(X_i-\bar{X}_{k,N})^2\ \ 
\text{where} \ \ 
\bar{X}_{k,N}=\frac{1}{N-k} \sum_{i=k+1}^N X_i.
\end{equation}
That is, this approach selects the number of initial samples $\hat k$ to set aside in order to minimize an estimate of $\sigma^2/(N-\hat k)$, which is (within a multiplicative constant) an estimate of the squared uncertainty, $\varepsilon_{N-\hat k}^2$, of the averaged value of the remainder of the dataset.  This optimization problem can be solved in $O(N^2)$ flops using a brute force method, or with $O(N)$ flops using a more advanced optimization algorithm (see, e.g., \cite{beyhaghi-2015, beyhaghi-2016}).

\clearpage
\section{Estimation of the averaging error}\label{sec:met}
In this section, we present a new multiscale technique to estimate the averaging error of a dataset modelled as a stationary ergodic random process $x_i$.  Define an additional random variable, referred to as the {\it sample mean} $y_s$, such that
\begin{gather*}
y_s= \frac{1}{s}\sum_{i=1}^s x_i,
\end{gather*}
and identify the {\it mean} $\mu$, {\it variance} $\sigma^2$, and {\it autocorrelation function} $\rho(k)$ such that
\begin{equation}\label{eq:mean,variance,autocorr}
\mu = \E [x_i] = \E[y_s], \quad
\sigma^2 = \E[(x_i-\mu)^2], \quad
\rho(k)=\frac{\E[(x_i-\mu) (x_{i+k}-\mu)]}{\sigma^2}.
\end{equation}
The {\it expected squared averaging error}, defined as the expected value of the square of the deviation of the sample mean $y_s$ from the true mean $\mu$, is given (see \S 1 of \cite{beran-1994}) by
\begin{equation} \label{eq:sigmaTrhok}
\varepsilon^2_s =\E[(y_s-\mu)^2]
=\frac{\sigma^2 }{s} \Big[1 + 2 \, \sum \limits_{k=1}^{s-1}  (1 - \frac{k}{s}) \, \rho(k) \Big].
\end{equation} 
Note that $\varepsilon^2_s\rightarrow 0$ as $s\rightarrow \infty$, as the process is ergodic, and that
\eqref{eq:sigmaTrhok} reduces to \eqref{eq:sigmaN} in the limit that the samples are i.i.d.
It follows immediately from \eqref{eq:sigmaTrhok} that
\begin{equation} \label{eq:MSanal}
	\E[y_s^2]=\mu^2 + \varepsilon_s^2 = \mu^2 +
	\frac{\sigma^2}{s}\, \Big[1 + 2 \, \sum \limits_{k=1}^{s-1}  (1 - \frac{k}{s}) \, \rho(k) \Big].
\end{equation}

Now consider a sequence of $N$ statistically stationary random variables, $x_1$ to $x_N$. An unbiased estimate of $\mu$ can be developed from this sequence leveraging the following definition of the {\it shifted sample means}
\begin{subequations}\label{eq:M}
\begin{equation}\label{eq:Ms}
m_{\ell,\ell+s}=\frac{1}{s}\sum_{i=\ell+1}^{\ell+s} x_i
\quad \textrm{for} \quad
\ell=0,1,\ldots,N-s,
\end{equation}
each of which is considered as a random variable with a distribution identical to that of $y_s$.  Note that the shifted sample means are not independent.  In practice, in the spirit of a batch means method, we will consider only those shifted sample means $m_{\ell,\ell+s}$ corresponding to non-overlapping blocks such that $\ell\in L_s=\{0, s, 2\,s, \dots, (p_s-1)\, s\}$ where $p_s= \floor{\frac{N}{s}}$. 
Define also the {\it mean-squared shifted sample mean}, $\bar{m}^2_s$, as the mean-squared value of $m_{\ell,\ell+s}$ for these $p_s$ nonoverlapping blocks,
\begin{equation}\label{eq:hatMs}
  \bar{m}^2_s = \frac{1}{p_s} \sum_{\ell\in L_s} m^2_{\ell,\ell+s};
\end{equation}
\end{subequations}
the random variable $\bar{m}^2_s$ has the same expected value as $y^2_s$, but reduced variance.

We will denote\footnote{That is, random variables in this work are indicated by lowercase letters, and corresponding realizations of these random variables are indicated by uppercase letters.} $\{X_1,X_2,\dots,X_N\}$ as a dataset modelled as a realization, of length $N$, of the random process $x_i$ described above. Corresponding realization values of the sample mean, the shifted sample means, and the mean-squared shifted sample mean are denoted by $Y_s$, $M_{\ell,\ell+s}$, and $\bar{M}^2_s$, respectively.  The value of $\bar{M}^2_s$ computed from this dataset provides an estimate of the expected mean of $y^2_s$ that is accurate for values of $s$ that are small enough that there are several blocks to average over; we thus only consider in this work the mean-squared shifted sample mean $\bar{m}^2_s$, and its realization value $\bar{M}^2_s$, for $s \le q_N$ where $q_N=\floor{\sqrt{N}}$.


We now define a model quantity $\hat\varepsilon^2_s$, in an analogous form as the expected squared averaging error $\varepsilon^2_s$ in \eqref{eq:sigmaTrhok}, such that
\begin{equation} \label{eq:sigmaTrhok1}
\hat\varepsilon^2_{s}
=\frac{\hat\sigma^2}{s} \Big[1 + 2 \, \sum \limits_{k=1}^{s-1}  (1 - \frac{k}{s}) \, \hat\rho(k;\hat\theta) \Big],
\end{equation} 
where $\hat\sigma$ is an model (i.e., an estimate) of the variance $\sigma$, and $\hat\rho(k;\hat\theta)$ is a model\footnote{Models of the autocorrelation function of various statistics of interest, in a number of chaotic systems of interest, have been studied broadly (e.g., autocorrelations of some statistics in turbulent flows are discussed in \cite{wynngaard-1973,kim1987turbulence}). The autocorrelation model  given in \eqref{eq:automodel} is typical in such studies.} of the autocorrelation function $\rho(k)$, with its adjustable model parameters assembled into the vector $\hat\theta$:
\begin{equation}\label{eq:automodel}
\hat\rho(k;\hat\theta)=\sum_{i=1}^m \hat A_i \, \hat\tau_i^k
\quad \textrm{where} \quad
\hat\theta=[\hat A_1, \hat A_2, \dots, \hat A_m, \hat\tau_1, \hat\tau_2, \dots, \hat\tau_m],
\end{equation}
where the feasible domain $\Omega$ for the $\hat\theta$ parameters is the linearly constrained domain
\begin{equation}\label{eq:Omega}
0\le \hat\tau_i < 1, \quad \sum_{i=1}^m \hat A_i=1.
\end{equation}

We now denote, by $\{\hat\theta_N,\hat\sigma_N,\hat\mu_N\}$, optimized values of $\{\hat\theta,\hat\sigma,\hat\mu\}$ based on the sequence $\{x_1,\ldots,x_N\}$ of length $N$.
These optimized values are determined by solving the following optimization problem:
\begin{subequations}\label{eq:modelfit}
\begin{gather}
\{\hat\theta_N,\hat\sigma_N,\hat\mu_N\}=
\textrm{argmin}\, f(\hat\theta,\hat\sigma,\hat\mu)= \sum_{s=1}^{q_N} \big[g_{s}(\hat\theta,\hat\sigma,\hat\mu)\big]^2 \quad \textrm{where} \label{eq:modelfit;a}\\
g_{s}(\hat\theta,\hat\sigma,\hat\mu) =  \hat\mu^2+ \frac{\hat\sigma^2 }{s} \Big[1 + 2 \, \sum \limits_{k=1}^{s-1}  \big(1 - \frac{k}{s}\big) \, \hat\rho(k;\hat\theta)\Big]- \bar{m}_s^2, \label{eq:modelfit;b}
\end{gather}
\end{subequations}
where $\bar{m}_s$ is derived from the sequence $\{x_1,\ldots,x_N\}$ via \eqref{eq:M}, while imposing that the last term in the sum in \eqref{eq:modelfit;a} vanishes, $g_{q_N}(\hat\theta_N,\hat\sigma_N,\hat\mu_N)=0$, in addition to the feasibility of the parameters of the autocorrelation model, $\hat\theta\in\Omega$ [see \eqref{eq:Omega}], as well as $\hat\sigma\ge 0$ and $\hat\mu\ge 0$.  In other words, we seek to find the best model parameters, $\{\hat\theta_N,\hat\sigma_N,\hat\mu_N\}$, such that the expression for $\E[y_s^2]$ in \eqref{eq:MSanal}, leveraging the model values $\hat\mu$ and $\hat\sigma$ and the model of the autocorrelation in \eqref{eq:automodel}, $\hat\rho(k;\hat\theta)$, exactly matches the unbiased estimate $\bar m_s^2$ of $y_s^2$ at $s=q_N=\floor{\sqrt{N}}$, while the sum of the squares of the mismatch of these quantities over all the batch lengths $1\le s <q_N$ is minimized.  That is, the tuning of the available parameters in the model, $\{\hat\theta,\hat\sigma,\hat\mu\}$, is performed in such a way as to accurately match the model [given in \eqref{eq:sigmaTrhok1}] of the expected squared averaging error [given by \eqref{eq:sigmaTrhok}], at a range of different timescales $s$, to the information available in the sequence $\{x_1,\ldots,x_N\}$; we thus refer to the approach developed as a {\it multiscale} fit.

In this work, for any given realization of the sequence, $\{X_1,\ldots,X_N\}$, the optimization problem defined by \eqref{eq:modelfit} is solved using SNOPT \cite{gill-2005}, which is an advanced Sequential Quadratic Programming (SQP) method.
Though the application of such a solver to a problem of this form is entirely straightforward, the optimization problem given in \eqref{eq:modelfit} is nonconvex, and thus SNOPT might only find a local minimum.  The resulting framework for estimating the averaging error is summarized in Algorithm \ref{algorithm:uq}.

\begin{algorithm}[t!]
	\caption{Estimation of the expected averaging error $\varepsilon^2_N=\E [(y_N-\mu)^2\}$ of a set of data $\{X_1,X_2,\dots,X_N\}$ modelled as a realization of a stationary random process $\{x_1,x_2,\dots,x_N\}$.}
	\label{algorithm:uq}
	\begin{algorithmic}[1]
		\For{each $s \in \{1,2, \dots, \floor{\sqrt{N}}\}$}
		\State Calculate $M_{\ell,\ell+s}$ via \eqref{eq:Ms} for all $\ell \in L=\{0,s, 2\, s, \dots, (\floor{\frac{N}{s}}-1) s\}$.
		\State Compute the mean-squared value $\bar M_s^2$ of the $M_{\ell,\ell+s}$ for all $\ell \in L$ via \eqref{eq:hatMs}.
		\EndFor
		\State Solve the optimization problem \eqref{eq:modelfit} to find the optimized model parameters $\{\hat\theta_N,\hat\sigma_N,\hat\mu_N\}$. 
		\State Generate an estimate of the averaging error, $\hat\varepsilon^2_N$, with \eqref{eq:sigmaTrhok1}, using the autocorrelation model given by $\hat\rho(k;\hat \theta)$ in \eqref{eq:automodel}, implementing the optimized parameter values $\{\hat\theta_N,\hat\sigma_N,\hat\mu_N\}$ determined in step 5.
	\end{algorithmic}
\end{algorithm}

Analytical expressions for the derivatives of $f(\hat\theta,\hat\sigma,\hat\mu) $ and $g_s(\hat\theta,\hat\sigma,\hat\mu)$ are useful in the optimization process.
The derivatives of the these functions are given as follows:
\begin{gather*}
\nabla f(\hat\theta,\hat\sigma,\hat\mu) =  2\, \sum_{s=1}^ {}  \, g_s(\hat\theta,\hat\sigma,\hat\mu) \, \nabla g_s(\hat\theta,\hat\sigma,\hat\mu), \\
\frac{\partial g_s(\hat\theta,\hat\sigma,\hat\mu) }{\partial \hat\mu}= 2\,\hat\mu, \quad \frac{\partial g_s(\hat\theta,\hat\sigma,\hat\mu) }{\partial \hat\sigma}= \frac{2\, \hat\sigma}{s} \Big[1+2\, \sum \limits_{k=1}^{s-1}  (1 - \frac{k}{s}) \, \sum_{i=1}^m \hat A_i \hat\tau_i^k \Big], \\ 
\frac{\partial g_s(\hat\theta,\hat\sigma,\hat\mu) }{\partial \hat A_i}= \frac{\hat\sigma^2}{s} \Big[2 \, \sum \limits_{k=1}^{s-1}  (1 - \frac{k}{s}) \, \hat\tau_i^k \Big], \quad 1 \leq i \leq m, \\ 
\frac{\partial g_s(\hat\theta,\hat\sigma,\hat\mu)}{\partial \hat\tau_i}= \frac{\hat\sigma^2}{s} \Big[2 \, \sum \limits_{k=1}^{s-1}  \hat A_i \, k \, (1 - \frac{k}{s})\, \hat\tau_i^{k-1} \Big], \quad 1 \leq i \leq m.
\end{gather*}

\section{Analysis of the estimator}\label{sec:analysis}

We now analyze various properties of the new estimation technique, presented in \S \ref{sec:met} and summarized in Algorithm \ref{algorithm:uq}, applied to a stationary random process $\{x_1,x_2,\dots\}$.  The mean $\mu$, variance $\sigma^2$, and autocorrelation $\rho(k)$ of the random process $x_i$ considered are defined as in \eqref{eq:mean,variance,autocorr}.

If the autocorrelation function $\rho(k)$ is summable, then it follows from \eqref{eq:sigmaTrhok} that the expected squared averaging error, $\varepsilon^2_s$, approaches zero like the reciprocal square root of $s$ times a constant $Q$, that is,
\begin{equation} \label{eq:summableq}
\lim_{s \rightarrow \infty } s\, \varepsilon_s^2 = Q
\quad \textrm{for finite } Q.
\end{equation}
It is natural to seek a UQ method that satisfies the same property; this notion is made precise by the following definition and theorem.
\vskip0.05in

\begin{definition}\label{def:asm_unb}
The random process $a_s$
is called an \emph{asymptotically  unbiased} estimate of the expected squared averaging error $\varepsilon_s$ if
\begin{equation}
\lim_{s \rightarrow \infty} s\,  (\E[a_s]-  \varepsilon_s^2) =0.
\end{equation}
\end{definition}
\begin{theorem}
The random process $\hat\varepsilon_s^2$ in \eqref{eq:sigmaTrhok1}, the $N$'th element of which is obtained by implementing Algorithm \ref{algorithm:uq} on the first $N$ elements of the random process $x_i$, provides an asymptotically unbiased estimate of $\varepsilon_s^2$ in \eqref{eq:sigmaTrhok}.
\end{theorem}

\begin{proof}
The model of the autocorrelation function \eqref{eq:automodel}, with parameters as optimized by Algorithm \ref{algorithm:uq}, is necessary summable (since the $\tau_i<1$ for all $i$). Denote $\hat\varepsilon_{s,N}^2$ as the model, given by \eqref{eq:sigmaTrhok1}, of the expected squared averaging error over a sequence of length $s$, implementing the optimized parameters $\{\hat\theta_N,\hat\sigma_N,\hat\mu_N\}$ derived from a sequence of length $N$.  It follows that
\begin{equation} \label{eq:summableqt}
\lim_{s \rightarrow \infty } s \, \hat\varepsilon_{s,N}^2 = \hat Q_{N},
\quad \textrm{for finite } \hat Q_N.
\end{equation}
By construction, the parameters $\hat\theta_N,\hat\sigma_N,\hat\mu_N$, $\hat Q_N$, and $\hat\varepsilon_{q_N,N}$ are random variables, as they are derived from the random process $x_i$. 
Moreover, these variables are obtained by solving the optimization problem \eqref{eq:modelfit}; therefore, the equality constraint of the optimization problem, $g_{q_N}(\hat\theta,\hat\sigma,\hat\mu)=0$ where $q_N=\floor{\sqrt{N}}$, must be satisfied:
\begin{equation} \label{eq:sigmatmhat}
\hat\mu_N^2+   \hat\varepsilon_{q_N,N}^2- \bar m_{q_N,N}^2=0.
\end{equation}
Since $\bar m_{q_N,N}^2$ is unbiased, $\E[\bar m_{q_N,N}^2]=\mu^2 +\varepsilon_{q_N}^2$.
Taking the expected value of \eqref{eq:sigmatmhat}, it follows that
\begin{equation}
\E[\hat \mu_N^2] + \E[\hat \varepsilon_{q_N,N}^2]
-\mu^2-\varepsilon^2_{q_N}=0.
\end{equation}
Multiplying the above equation by $q_N$ and rearranging gives
\begin{equation}
q_N(\E[\hat\mu_N^2]-\mu^2)+
(\E[q_N \hat \varepsilon_{q_N,N}^2]-q_N \varepsilon_{q_N}^2)=0.
\end{equation}
Thus, taking the limit as $N\rightarrow\infty$ (and, therefore, $q_N\rightarrow\infty$) and substituting \eqref{eq:summableq} and \eqref{eq:summableqt}, it follows that
\begin{equation*}
\lim_{N \rightarrow \infty } q_N (\E[\hat \mu_N^2]-\mu^2)+ \E[\hat Q_N]-Q=0.
\end{equation*}
Since $Q$ and $\hat Q_N$ are bounded,
\begin{gather*}
\lim_{N \rightarrow \infty } \E[\hat \mu_N^2]=\mu^2, \quad \lim_{N \rightarrow \infty } \E[\hat Q_N]=Q, \notag\\
\lim_{N \rightarrow \infty} N(\E[\hat\varepsilon_N^2]-\varepsilon_N^2)= \lim_{N \rightarrow \infty} \E[\hat Q_N]-Q=0. \qedhere
\end{gather*}

\end{proof}

We have thus established that the implementation of Algorithm \ref{algorithm:uq} on any realization of a stationary random process $x_i$ results in an unbiased estimate of the averaging error. This is a valuable property of the present method for estimating the averaging error, as it implies that the uncertainty quantification does not have any systematic error in the limit of large $N$. 
Note that certain other leading methods for uncertainty quantification, such as that developed in \cite{beran-1994}, based on a maximum likelihood formulation, do not share this valuable property. 

\clearpage
\section{Numerical simulations} \label{sec:result}

We now apply the algorithm developed in section \ref{sec:met} for estimating the averaging error to three different datasets generated as follows:
\begin{itemize} 
\item[1.] A synthetic autoregressive process of order six, AR(6). 
\item[2.] Statistics of the kinetic energy of the Kuramoto-Sivashinsky (KS) equation.
\item[3.] Statistics of the TKE from a DNS of a turbulent channel flow at $Re_{\tau}=180$. 
\end{itemize}

For the purpose of comparison, in all three cases, the averaging error is estimated using the maximum likelihood approach.

Also, for comparison, the expected squared averaging error $\varepsilon_s^2$, given by \eqref{eq:sigmaTrhok}, is calculated based on accurate values of $\mu$, $\sigma^2$, and $\rho(k)$. For the dataset generated by AR(6), the true values of $\mu$, $\sigma^2$, and $\rho(k)$ are available, so the true value of $\varepsilon_s^2$ is directly computable. 

On the other hand, since analytical expressions for $\mu$, $\sigma^2$, and $\rho(k)$ are not available in KS and NSE cases, a ``truth model'' for the UQ of the these datasets are found based on the data available.  To develop a ``truth model'' for the UQ of KS dataset, we simply approximate the autocorrelation, $\rho(k)$, from a very large $N$ (at least 100 times larger than the maximum value of $N$ considered in the plots) and estimate the expected squared averaging error $\varepsilon_s^2$, given by \eqref{eq:sigmaTrhok} and approximated values of $\mu$, $\sigma^2$, and $\rho(k)$. Moreover, to develop a ``truth model'' for the UQ of the NSE dataset, since direct approximation of $\rho(k)$ is computationally intractable, we simply apply the algorithm developed above for very large $N$ (at least 30 times larger than the maximum value of $N$ considered in the plots).  This approach provides a very large number of samples to average over when approximating $\varepsilon_s^2$ numerically.

\subsection{Autoregressive model}\label{sec:ARmodel}

We first apply the new UQ method developed in \S \ref{sec:met} to a dataset generated by an autoregressive (AR) model of the general form
\begin{equation}\label{eq:AR6}
x_i=\sum_{k=1}^{n} \alpha_k \, x_{i-k}+\epsilon_i;
\end{equation}
for the present study, we take $\epsilon_i=\mathcal{N}(0,\sigma^2_\epsilon)$.  After an initial transient (related to the specified $n$ initial values of $x_i$) has passed, this system is statistically stationary, with a mean of $\mu=0$ and, defining the unnormalized autocorrelation function $\gamma(k)=\sigma^2\,\rho(k)=\E\{(x_i-\mu) (x_{i+k}-\mu)\}$, the values of $\gamma(k)$ related (see \cite{beran-1994}) by
\begin{subequations}\label{eq:ARstats}
\begin{align}
\gamma(0)&=\sum_{j=1}^{n} \alpha_j \gamma(j) + \sigma^2_\epsilon,\label{eq:ARstats;a} \\
\gamma(h)&=\sum_{j=1}^{n} \alpha_j \gamma(h-j) \quad \textrm{for} \quad h>0.\label{eq:ARstats;b}
\end{align}
\end{subequations}
Noting that $\gamma(k)=\gamma(-k)$, the first $n+1$ values of $\gamma$ may be determined by solving the linear system of equations, known as the Yule-Walker equations, given by \eqref{eq:ARstats;a} together with \eqref{eq:ARstats;b} for $h=1,\ldots,n$; additional values of $\gamma(k)$ are then given directly by \eqref{eq:ARstats;b}.
Note further that $\sigma^2=\gamma(0)$, and $\rho(k)=\gamma(k)/\sigma^2$.

In the simulations reported here, we consider an AR(6) process (that is, we take $n=6$) with $\epsilon_i=\mathcal{N}(0,\sigma^2_\epsilon)$ where $\sigma^2_\epsilon=0.1$, and
\begin{gather*}
\alpha_1=3.1378,\ \alpha_2=-3.9789,\ 
\alpha_3=2.6788,\ \alpha_4=-1.0401,\ 
\alpha_5=0.2139,\ \alpha_6=-0.0133.
\end{gather*}
The poles of the difference equation corresponding to this AR(6) model are given by
\begin{equation*}
\big\{0.1,\ 0.95,\ 0.8,\ 0.7,\ 0.25+{\sqrt{3}}/{2},\ 0.25-{\sqrt{3}}/{2}\big\}.
\end{equation*}
After the initial transient has passed, the system in \eqref{eq:ARstats} reveals that the standard deviation $\sigma=24.97$, and that the autocorrelation function $\rho(k)$, for $k=0,1,2,\ldots$, is:
\begin{equation*}
\rho(k)=\{1,\ 0.9967,\ 0.9870,\ 0.9716 ,  \ 0.9516,\ 0.9277,\ 0.9010,\ 0.8722,\ 0.8418,\ldots\}.
\end{equation*}
The typical behavior of the AR(6) model described above is illustrated in Figure \ref{fig:avg_AR}.

\begin{figure}
    \centerline{
    \begin{subfigure}[b]{2.7in}
    \includegraphics[width=2.7in]{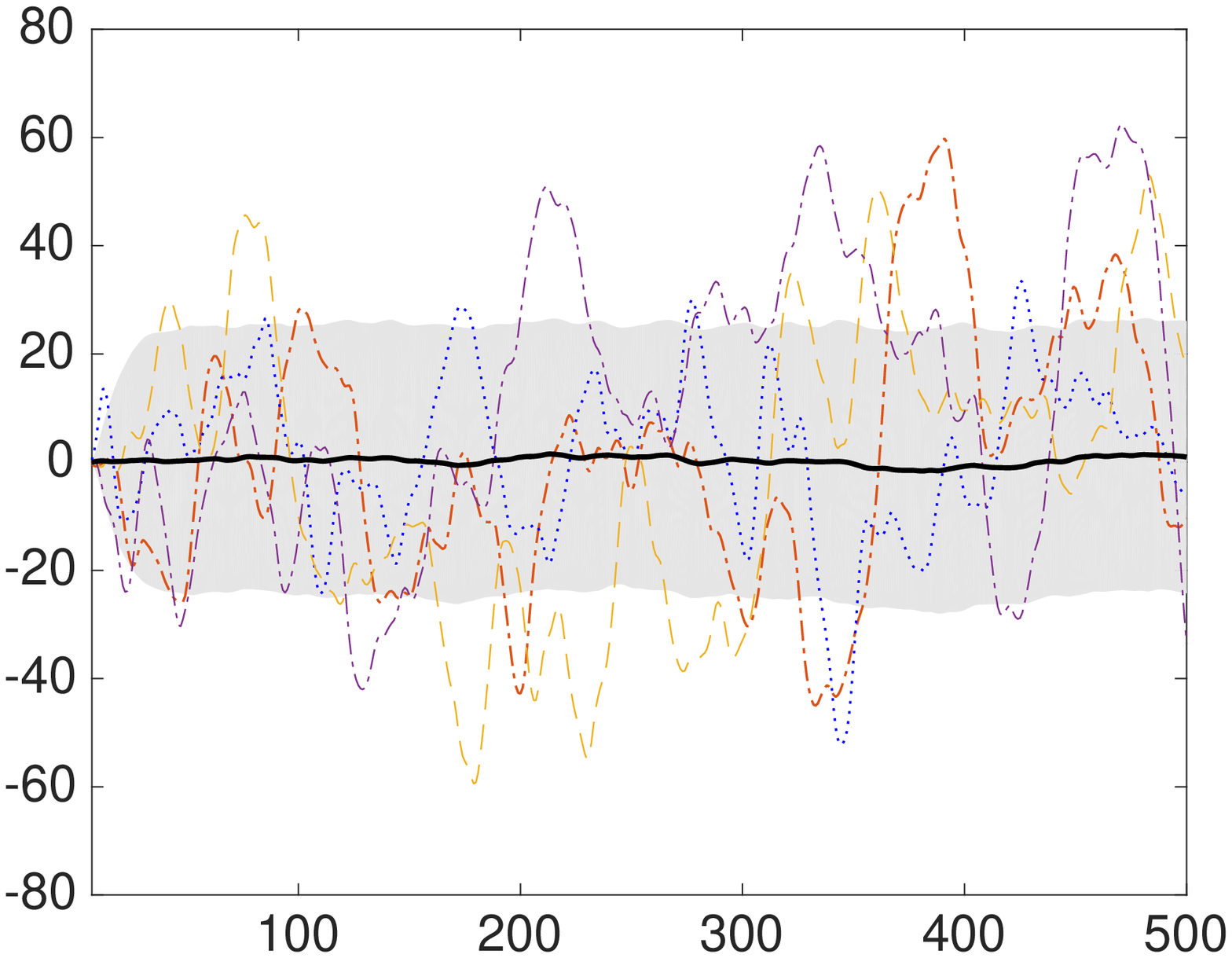}\vskip-0.10in
    \label{fig:mean_signal_0}
    \caption{Initialization: $x_{-5}=\ldots=x_0=0$.}
        \end{subfigure}
    \begin{subfigure}[b]{2.7in}
    \includegraphics[width=2.7in]{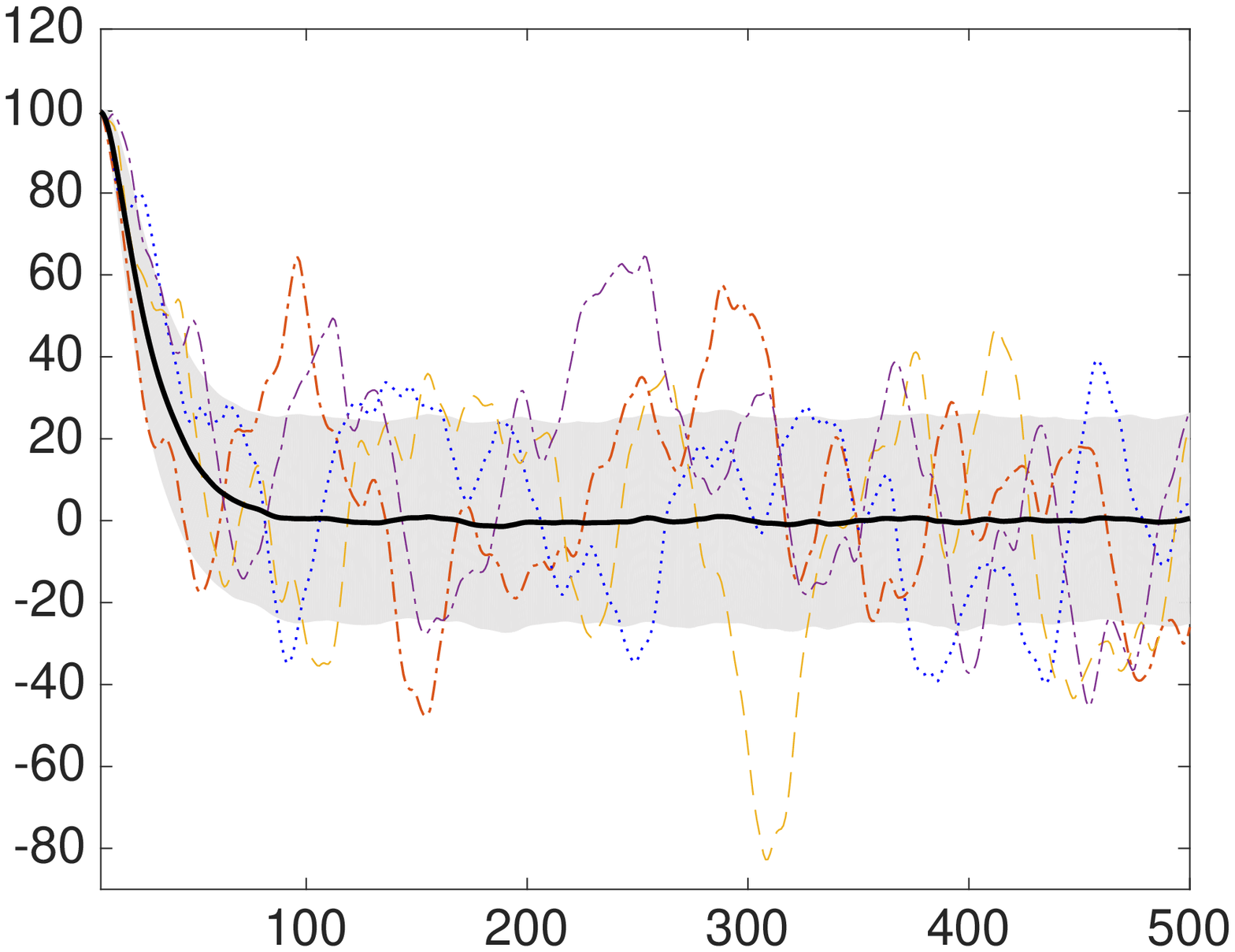}\vskip-0.10in
    \label{fig:mean_signal_100}
    \caption{Initialization: $x_{-5}=\ldots=x_0=100$.}
        \end{subfigure}}\vskip-0.1in
\caption{Simulation of the AR(6) model described in \S \ref{sec:ARmodel}, evolving away from different initial values of $x_i$ as indicated. 
(dot-dashed lines) Representative realizations. (solid line) Ensemble mean and (gray region) ensemble mean $+/-$ ensemble variance, computed over 1000 ensemble members.}
	\label{fig:avg_AR}
        
\centerline{
    \begin{subfigure}[b]{2.7in}
    \includegraphics[width=2.5in]{HIST_0_AR}\label{fig:HIS_TR_AR_0}\caption{Initialization as in Figure \ref{fig:avg_AR}a.}
        \end{subfigure}
    \begin{subfigure}[b]{2.7in}
    \includegraphics[width=2.5in]{HIST_100_AR}\label{fig:HIS_TR_AR_100}\caption{Initialization as in Figure \ref{fig:avg_AR}b.}
        \end{subfigure}}\vskip-0.1in
   \caption{Histogram of the transient time estimates (see \S \ref{sec:trans}) for the AR(6) model described in \S \ref{sec:ARmodel}, computed over 1000 ensemble members.}
	\label{fig:trans_AR}
\end{figure}

Figure \ref{fig:trans_AR} illustrates the behavior of the transient time estimation method reviewed in \S \ref{sec:trans} on the ensemble of 1000 simulations summarized in Figure \ref{fig:avg_AR}.  Initialization with $x_{-5}=\ldots=x_0=0$, as illustrated in Figure \ref{fig:trans_AR}a, shows that, for about 75\% of the ensemble members considered, the minimization problem given by \eqref{eq:trans} results simply in $\hat k=1$.  This is entirely to be expected, as the transient estimation method implemented in this work is {\it not} based on a second-order stationarity condition, but rather simply on the minimization of the average squared value of $(X_i-\bar X)$ over the remaining samples.  Thus, even though the AR(6) system is clearly not statistically stationary in the first few samples in this case, the particular transient estimation method implemented is insensitive to this fact.

Initialization with $x_{-5}=\ldots=x_0=100$, as illustrated in Figure \ref{fig:trans_AR}b, shows a much more typical behavior of the transient time estimation method selected for the problems of interest in this work.  In this case, the minimization problem given by \eqref{eq:trans} results in an average value of $\hat k=40$, which is very nearly the value one would select by eye given the (very significant) advantage of hindsight, as embodied by the ensemble average results depicted in Figure \ref{fig:avg_AR}b.  This is indeed remarkable, as each transient time calculation is based solely on an individual ensemble member, each of which has significant random fluctuations associated with it (see the dot-dashed curves of Figure \ref{fig:avg_AR}b).

Figure \ref{fig:ar_test} illustrates the performance of Algorithm \ref{algorithm:uq} on an ensemble of $30$ datasets obtained via simulation of the AR(6) model described above, taking $x_{-5}=\ldots=x_0=0$ and $\hat k=0$, with realization lengths of $N=\{2^{7:14}\}$.  For the purpose of comparison, we have also estimated the averaging error using the maximum likelihood approach applied to an AR(3) model; to facilitate a fair comparison, the same number of model parameters is used for both methods.  We also compare with the expected squared averaging error given by \eqref{eq:sigmaTrhok}, with the exact formulae for $\mu$, $\sigma$, and $\rho$ as determined by solution of \eqref{eq:ARstats}.

It is clearly evident for moderate to large realization lengths, $N\gtrsim 2000$, that the performance of the uncertainty quantification method given by Algorithm \ref{algorithm:uq} is remarkably better than that given by the MLE-based approach.  In particular, it is distinctly evident that the estimates given by Algorithm \ref{algorithm:uq} are asymptotically unbiased (see Definition \ref{def:asm_unb}), whereas the estimates generated by the MLE-based approach are not (as the underlying AR(3) model used in the MLE-based approach can not entirely capture the dynamics of the AR(6) system).  For small realization lengths $N$, the performance of the estimators are similar.

\begin{figure}
	\centering
	\includegraphics[width=1\columnwidth]{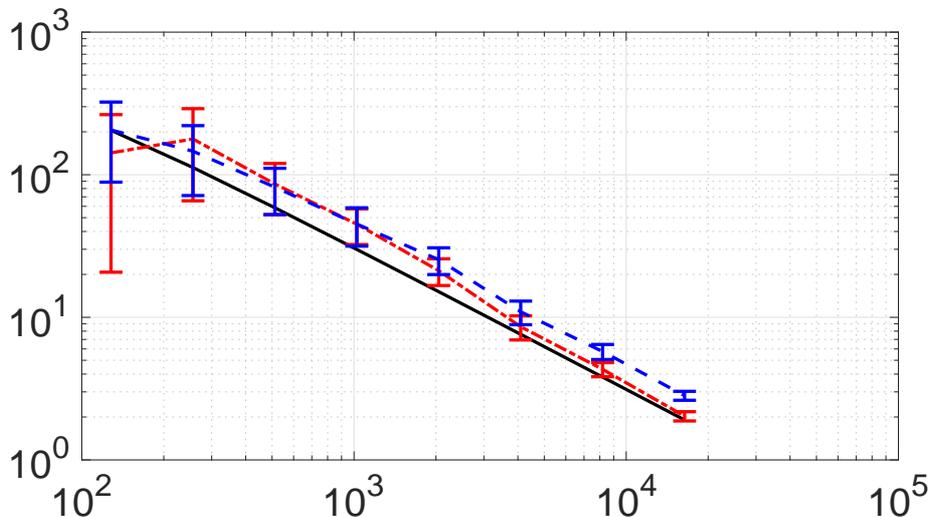}\vskip-0.01in
	\caption{Implementation of Algorithm \ref{algorithm:uq} and the MLE-based UQ approach on an ensemble of 30 simulations of the AR(6) model described in \S \ref{sec:ARmodel}. (horizontal axis) Averaging length $N$, (vertical axis) averaging error $\varepsilon_N$, (red dotted-dashed line) Ensemble average, (red error bars) ensemble variance of the estimate of the averaging error given by Algorithm \ref{algorithm:uq}, (blue dashed line) Ensemble average and (blue error bars) ensemble variance of the estimate of the averaging error given by the MLE-based approach, using an AR(3) model, (solid black line) actual averaging error, based on \eqref{eq:sigmaTrhok} with the true values of $\mu$, $\sigma^2$, and $\rho(k)$.}
	\label{fig:ar_test}
\end{figure}    

\subsection{Kuramoto-Sivashinsky equation}\label{sec:KS}

We next apply the new UQ method to a dataset derived from a simulation of the Kuromoto-Sivashinsky (KS) equation,
\begin{equation}\label{eq:KS}
u_t+u_{xxxx}+u_{xx}+u\,u_x=0 \quad \textrm{for} \quad 0\le x<L,
\end{equation}
with periodic boundary conditions $u(0)=u(L)$.  The statistic we consider in this work is the spatially-averaged value of the energy, defined as
\begin{equation}
e(t)=\frac{1}{L} \int_0^{L} u^2(x,t)\,dx.  
\end{equation}
The KS PDE is simulated in this work using a dealiased pseudospectral method for computing spatial derivatives, and a low-storage Implicit-Explicit Runge-Kutta (IMEXRK) scheme \cite{cavaglieri-2015, cavaglieri-2013} for marching in time. 

In the simulations reported here, we take $L=200$, $N_x=512$, $\Delta x=L/N_x\approx 0.391$, and $\Delta t=0.2$.  The initial field is taken as 
\begin{equation}
u(x,0)=\sin(0.5\, \pi\, \,x ) + \sin(0.85\, \pi \,x ) + 0.2\, v, \quad v=\mathcal{N}(0,1). 
\end{equation}  
After the initial transient has passed, the KS system defined above approaches a chaotic attractor, as indicated in Figure \ref{fig:transtime}a.

The transient identification method reviewed in \S \ref{sec:trans} is again implemented to detect and set aside the initial transient in the dataset.  A typical estimate of the transient time is illustrated in Figure \ref{fig:transtime}; after setting aside this initial transient, the remainder of the dataset appears to be approximately statistically stationary.

\begin{figure*}
			\centering
\includegraphics[width=0.7\linewidth]{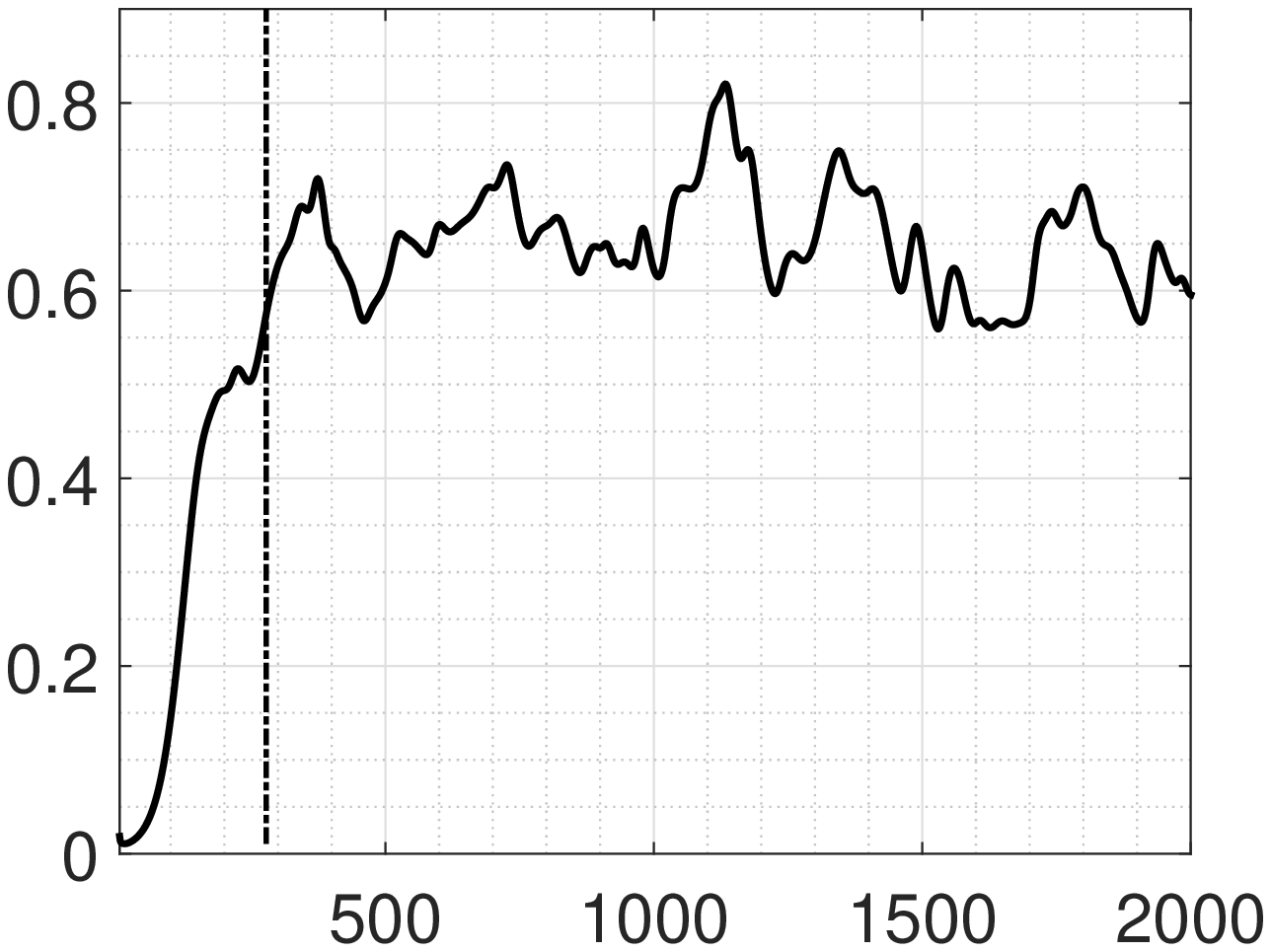} 
	
	\caption{Evolution of (solid line) the kinetic energy in a simulation of the KS model described in \S \ref{sec:KS} versus the number of timesteps taken, including (vertical dashed line) the identification of the initial transient.} 
	\label{fig:transtime}
    
	\centering
	\includegraphics[width=0.7\columnwidth]{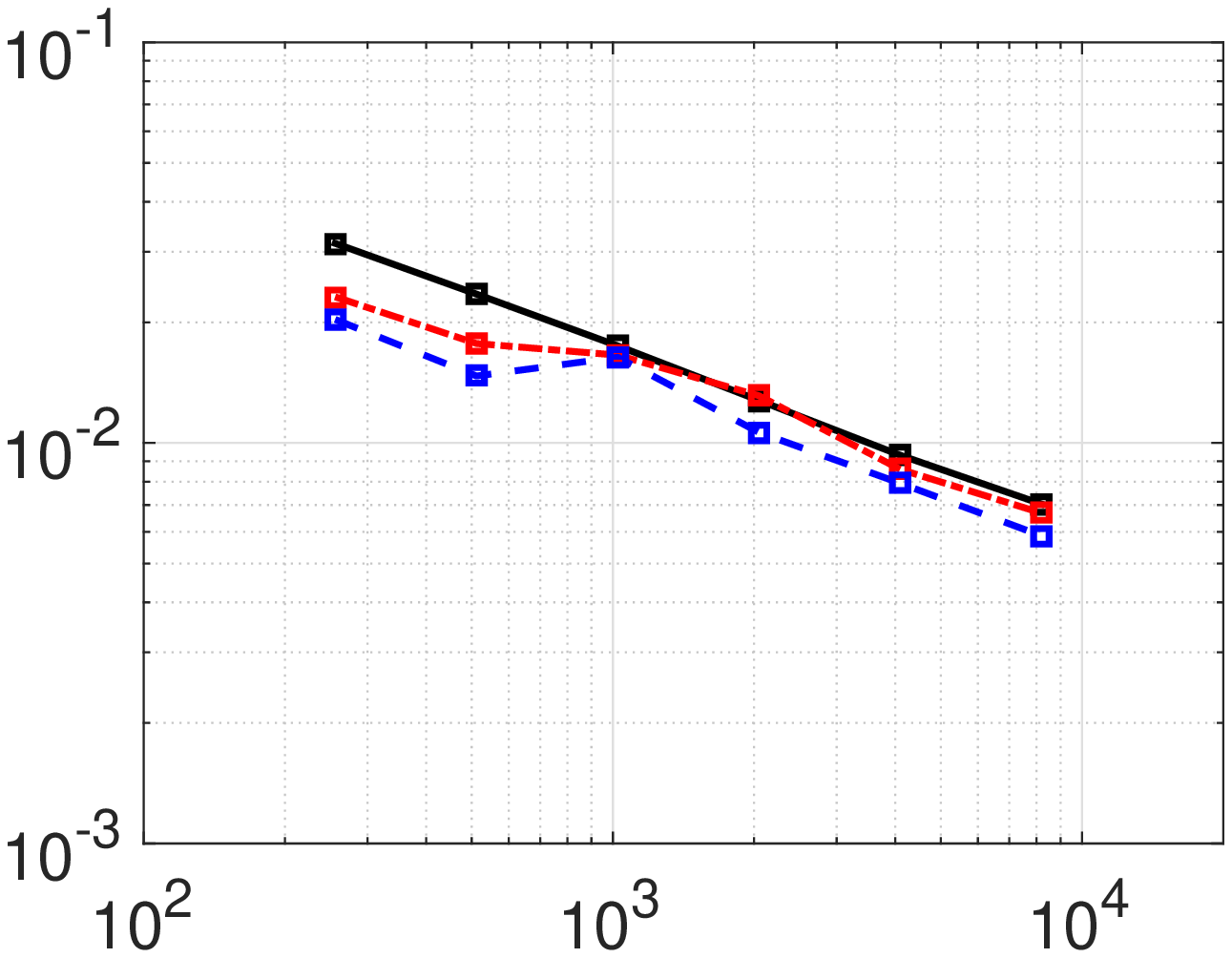}\vskip-0.01in
	\caption{Implementation of Algorithm \ref{algorithm:uq} and the MLE-based UQ approach on a single simulation of the KS model described in \S \ref{sec:KS}.  Horizontal axis is the number of samples taken, with one sample per timestep.  Line types are identical to those in Figure \ref{fig:ar_test}, with the error bars removed because only a single simulation is shown.  Note that accurate estimates, not exact values, are used for $\mu$, $\sigma^2$, and $\rho(k)$ when computing the actual averaging error.}
	\label{fig:kuromoto}
\end{figure*}


Sampling every timestep after the initial transient (see Figure \ref{fig:transtime}) is set aside, the averaging error was computed with realization lengths of $N=\{2^8: 2^{13}\}=\{256:8192\}$, incrementing by powers of two.  Note that, since we take $\Delta T=0.2$, this is equivalent to taking $T=0.2*N=\{50:3268\}$ time units of the original KS equation.  Again, the averaging error is estimated using Algorithm \ref{algorithm:uq} and the MLE-based approach, with an AR(3) model incorporated.  We also compare with the expected squared averaging error given by \eqref{eq:sigmaTrhok}, with the accurate values for $\mu$, $\sigma$, and $\rho$ as determined from a simulation 30 times longer than the the longest simulation reported here.

As observed in Figure \ref{fig:kuromoto}, the performance of the UQ method developed here is significantly improved as compared with the MLE-based approach, especially as the number of samples is increased.

\subsection{Navier-stokes equations}\label{sec:NS}

Finally, we apply the UQ method to a dataset generated by a DNS of a low Reynolds number incompressible 3D turbulent channel flow (see, e.g., \cite{moin1998direct,moser1999direct}).
Periodic boundary conditions are applied in the streamwise direction, $x$, and the spanwise direction, $z$; homogeneous Dirichlet boundary conditions on the velocity are applied at the walls in the wall-normal direction, $y$.

Following \cite{kim1987turbulence}, the incompressible Navier-Stokes equation is implemented in a 2-variable formulation of the wall-normal components of velocity and vorticity (other velocity, vorticity, and pressure components may be computed from these two components as needed).  The simulation, which used the code developed in \cite{luchini2006low}, used a dealiased pseudospectral method for computing spatial derivatives in the $x$ and $z$ directions, and the compact finite difference method \cite{kwok2001critical} for computing spatial derivatives in the $y$ direction.  The CN/RKW3 method \cite{kim2003control} was used for time integration.  

In the simulations reported here, we consider a spatial domain with
$0\le x \le 2 \pi$,  $0 \le y \le 2$, and $0 \le z \le 2 \pi$,
a grid of $N_x=128$, $N_y=64$, and $N_z=128$, Reynolds number $Re_{\tau}=180$, and timesteps of $\Delta t=0.01$.
The simulation is performed for $N=10^5$ timesteps, 
and the statistic that is analyzed is turbulent kinetic energy (TKE).

As indicated in Figure \ref{fig:transient_ns}, the transient identification process is completely analogous to that in the KS case.  

Sampling every timestep after the initial transient (see Figure \ref{fig:transient_ns}) is set aside, the averaging error was computed with realization lengths of $N=\{2^9: 2^{12}\}=\{512:4096\}$, incrementing by powers of two.  Note that, since we take $\Delta T=0.01$, this is equivalent to taking $T=0.01*N=\{5.12:40.96\}$ time units of the original NS equation.  Again, the averaging error is estimated using Algorithm \ref{algorithm:uq} and the MLE-based approach, with an AR(18) model incorporated.  We also compare with the expected squared averaging error given by \eqref{eq:sigmaTrhok}, with the accurate values for $\mu$, $\sigma$, and $\rho$ as determined from a simulation much longer than the the longest simulation reported here.




Again, as observed in Figure \ref{fig:ns_averag}, the performance of the UQ method developed here is seen to be significantly improved as compared with the MLE-based approach, especially as the number of samples is increased.


\begin{figure*}[t!]
			\centering
\includegraphics[width=0.7\linewidth]{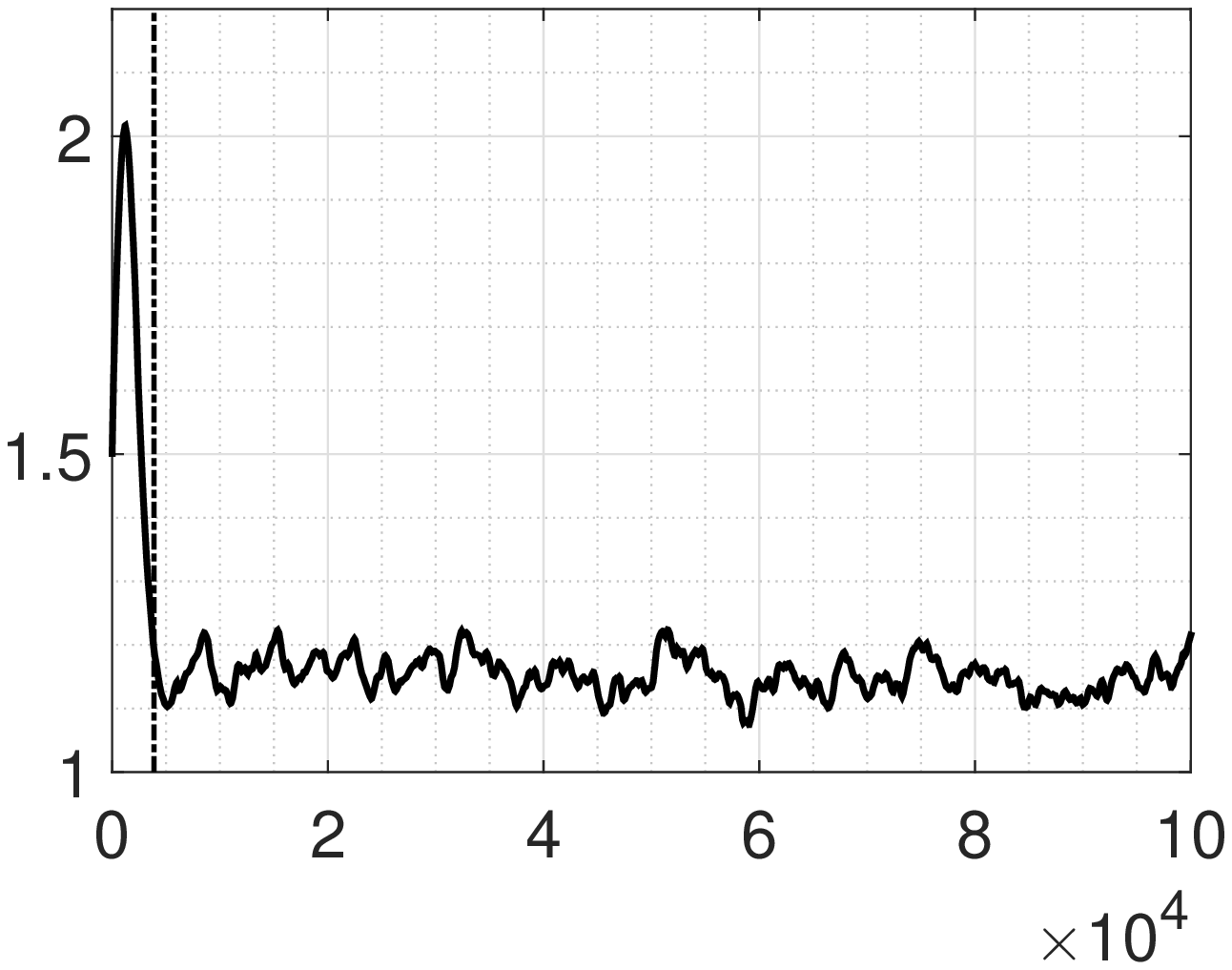} 
	\caption{Evolution of (solid line) the turbulent kinetic energy in a simulation of the $Re_\tau=180$ channel flow model described in \S \ref{sec:NS} versus the number of timesteps taken, including (vertical dashed line) the identification of the initial transient.} 
	\label{fig:transient_ns}
	\centering
		\includegraphics[width=.7\columnwidth]{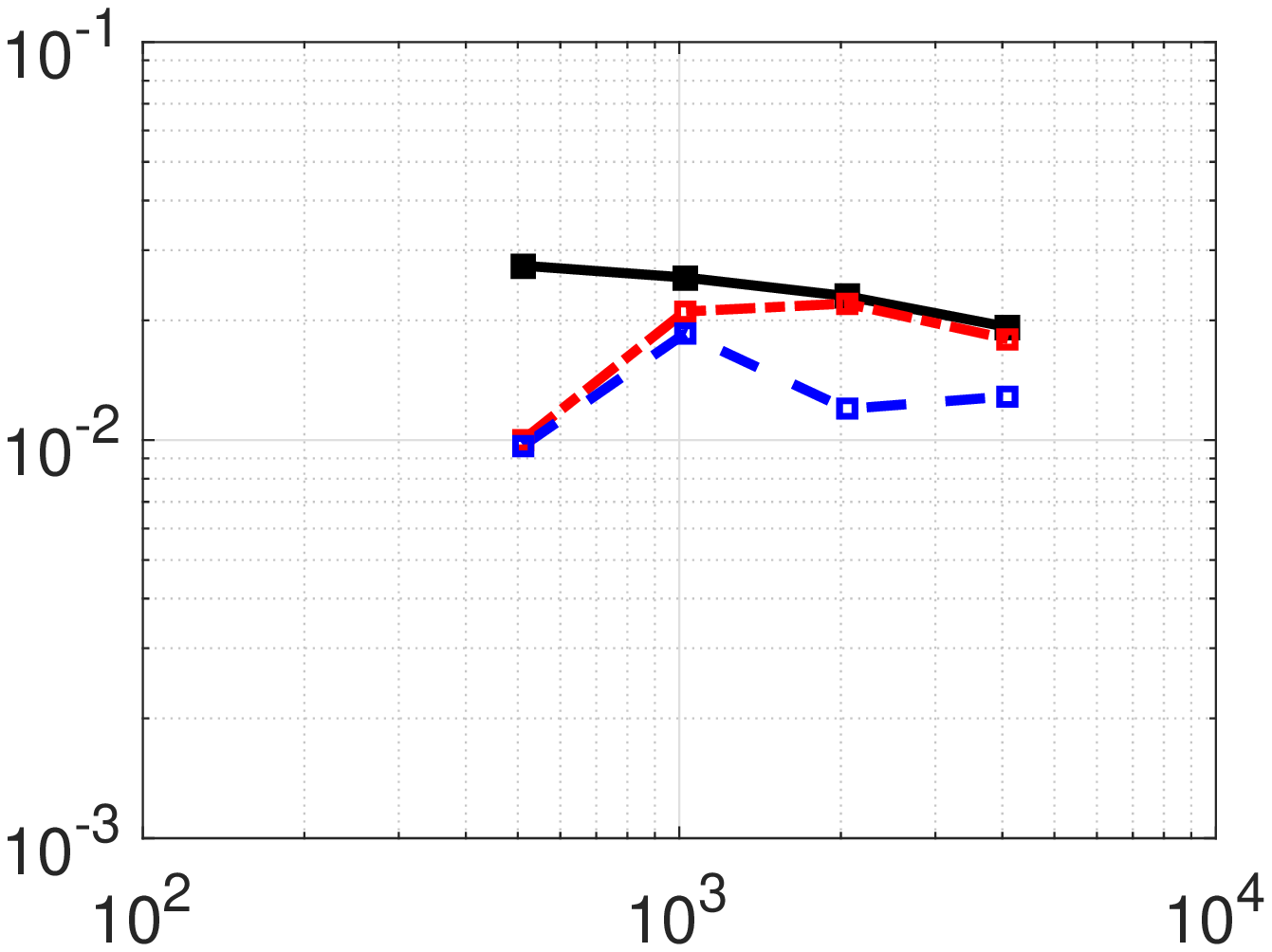}\vskip-0.01in
	\caption{Implementation of Algorithm \ref{algorithm:uq} and the MLE-based UQ approach on a single simulation of the $Re_\tau=180$ channel flow model described in \S \ref{sec:NS}.  Horizontal axis is the number of samples taken, with one sample per timestep.  Line types are identical to those in Figures \ref{fig:kuromoto} and \ref{fig:ar_test}.}
	\label{fig:ns_averag}
\end{figure*}

\section{Conclusions}

A new approach has been developed to quantify the uncertainty associated with finite-time-average approximations of infinite-time-average statistics of statistically stationary ergodic processes.  For applications of this new UQ approach that are derived from continuous-time chaotic systems like turbulent flows, an adequate sampling interval $h$ is identified in \eqref{eq:optimal_h}, and an effective method for removing the initial transient from the dataset is reviewed in \S \ref{sec:trans}.  A companion paper \cite{beyhaghi-alpha} illustrates how an effective UQ approach of this sort can be directly leveraged for maximally-efficient derivative-free optimization of infinite-time-averaged statistics of chaotic systems which depend upon a handful of adjustable parameters.

The new UQ method is presented in \S \ref{sec:met} and analyzed mathematically in \S \ref{sec:analysis}. This analysis reveals that, for long simulations, the UQ so determined is asymptotically unbiased; this important property is not guaranteed by various competing UQ methods, such as the leading method developed in \cite{beran-1994}, which based on a maximum likelihood formulation.

The new UQ method is tested in \S \ref{sec:result} on datasets generated by an AR(6) process, by the Kuramoto-Sivashinsky equation, and by a low-Reynolds number turbulent channel flow DNS.  Results are compared with both the leading UQ approach developed in \cite{beran-1994}, as well as the expected deviation of the sample mean $y_s$ from the true mean $\mu$, as quantified by \eqref{eq:sigmaTrhok}, based on accurate values of the true mean $\mu$, the variance $\sigma^2$, and the autocorrelation $\rho(k)$. 

It is observed that the method developed here has a significant improvement from that provided by the approach in \cite{beran-1994}, especially as the realization length $N$ is increased. An open source, Python implementation for this time averaging uncertainty quantification and its transient detector are available at  \url{https://github.com/salimoha/uq.git}. 

\section*{Acknowledgment}
The authors gratefully acknowledge Prof. Juan Carlos Del Alamo, Prof. Paulo Luchini, and Prof. Dariush Divsalar for their assistance, and funding from AFOSR FA 9550-12-1-0046, Cymer Center for Control Systems \& Dynamics, and Leidos corporation in support of this work. 

\bibliographystyle{siam}

\bibliography{ref}
\end{document}